\documentclass[12pt,a4paper]{amsart}
\usepackage{amssymb}
\usepackage{amsfonts}
\usepackage{amsmath}
\usepackage[mathscr]{eucal}

\usepackage{color}

\newtheorem{thm}{Theorem}

\newtheorem{cor}[thm]{Corollary}

\newtheorem{lem}[thm]{Lemma}

  \makeatletter

\@addtoreset{equation}{section}
\makeatother
\renewcommand{\a}{\alpha}

\newcommand{\go}{{\mathfrak{o}}}
\newcommand{\gp}{{\mathfrak{p}}}

\newcommand{\gX}{{\mathfrak{X}}}

\newcommand{\Acal}{{\mathcal A}}

\newcommand{\Fcal}{{\mathcal F}}

\newcommand{\Hcal}{{\mathcal H}}

\newcommand{\Ucal}{{\mathcal U}}

\newcommand{\CC}{\mathbb{C}}

\newcommand{\NN}{\mathbb{N}}

\newcommand{\QQ}{\mathbb{Q}}
\newcommand{\RR}{\mathbb{R}}
\newcommand{\ZZ}{\mathbb{Z}}

\newcommand{\bfK}{{\mathbf K}}

\newcommand{\GL}{\operatorname{GL}}

\newcommand{\PGL}{\operatorname{PGL}}

\newcommand{\End}{\operatorname{End}}

\newcommand{\bsl}{\backslash}

\title{A quantum probabilistic approach to Hecke algebras for $\gp$-adic $\PGL_2$}

\author{Takehiro Hasegawa}

\author{Hayato Saigo}

\author{Seiken Saito}

\author{Shingo Sugiyama}

\date{}
\begin{document}

\begin{abstract}
The subject of the present paper
is an application of quantum probability to $p$-adic objects.
We give a quantum-probabilistic interpretation of the spherical Hecke algebra
for $\PGL_2(F)$, where $F$ is a $p$-adic field.
As a byproduct, we obtain a new proof of the Fourier inversion formula for $\PGL_2(F)$.
\end{abstract}

\maketitle

\section*{Introduction}

Hecke operators acting on the space of modular forms are central objects in number theory.
There are stochastic problems of eigenvalues of Hecke operators.
As a remarkable example, the Sato-Tate conjecture (proved in \cite{HSBT}, \cite{Gee} and \cite{BLGHT}) is known. 
This states that, eigenvalues of all Hecke operators $\{T_p\}_p$ with respect to a fixed simultaneous eigenvector
are equidistributed in the interval $[-2,2]$ with respect to the probability measure
$$\frac{\sqrt{4-x^2}}{{2\pi}}dx.$$
Before the Sato-Tate conjecture was resolved,
as a deduction from the Eichler-Selberg trace formula,
Serre \cite{Serre} had found out the vertical Sato-Tate law:
For a fixed $p$-th Hecke operator $T_p$ for a prime number $p$, all eigenvalues of $T_p$ are equidistributed in the interval $[-2,2]$ with respect to the probability measure
$$\frac{p+1}{(p^{1/2}+p^{-1/2})^2-x^2}\frac{\sqrt{4-x^2}}{2\pi}dx.$$
Each of two measures as above is called the Wigner semicircle distribution and
the Kesten distribution with parameter $p+1$, respectively.
The Kesten distribution with parameter $p+1$ is understood as the spherical Plancherel measure for $\PGL_2(\QQ_p)$ and appears in studies of modular forms and their $L$-functions (e.g., \cite{RamakrishnanRogawski}, 
\cite{FeigonWhitehouse}, \cite{TsuzukiSpec}, \cite{Sugiyama2}, \cite{SugiyamaTsuzukiHolo}, \cite{SugiyamaTsuzukiDeriv} and \cite{SugiyamaTsuzukiJZ}).

The Wigner semicircle distribution and the Kesten distributions have been studied in a framework of quantum probability theory,
which captures a classical random variable as an element of a unital $\ast$-algebra not necessarily commutative.
A quantum probability space is a pair $(\Acal, \varphi)$ of a unital $\ast$-algebra $\Acal$ and a state $\varphi:\Acal \rightarrow \CC$.
The state $\varphi$ is a substitute of expectation,
and $(\Acal, \varphi)$ recovers classical probability theory when $\Acal$ is commutative.
It is well-known that the $m$-th moment $\varphi(a^m)$ for $a\in\Acal$ is described by
that of a probability measure and an interacting Fock space (cf.\ \cite{AB}), on which
an analogue of decomposition of the quantum harmonic oscillator into the creation and the annihilation operators are discussed.
Such a decomposition is applied
to analysis of asymptotic spectral distributions of graphs (cf.\ \cite{AccardiObata}, \cite{HoraObata} and \cite{Obata})
and
to quantum-classical correspondence asymptotically controlled by the arcsine law (cf.\ \cite{Saigo} and \cite{SaigoSako}).

The consideration above inspires us to give a new insight into the Hecke operator $T_p$ in terms of quantum probability theory.
In this paper, we treat the $p$-adic group $G=\PGL_2(\QQ_p)$ for a prime number $p$ and its spherical Hecke algebra $\Hcal(G,K)$
consisting of all $\PGL_2(\ZZ_p)$-biinvariant $\CC$-valued functions on $G$ with compact support, keeping in mind that the Hecke operator $T_p$ on modular forms corresponds to an action to adelized modular forms of an element $T(p\ZZ_p)$ of $\Hcal(G, K)$ (cf.\ \cite[\S3.6]{Bump}).
Our new contribution in this paper is to give an interacting Fock space structure to $\Hcal(G,K)$ and a quantum decomposition of the $T(p\ZZ_p)$-multiple (see Theorem \ref{Hecke alg IFS}).
As a corollary, we give a new proof of the Fourier inversion formula for $\PGL_2(\QQ_p)$ originally given by Macdonald \cite{Macdonald}.
The key of his proof is a concrete calculation of integrals concerned with the explicit formula \eqref{explicit of omega} of spherical function.
In contrast to his proof, we use quantum probability theory and intrinsic properties of spherical functions, and prove the inversion formula without explicit forms of those functions.
Throughout this paper, we describe our results for any non-archimedean local field $F$ in place of $\QQ_p$.


\section{Preliminaries}
Let $\NN$ denote the set of all positive integers and set $\NN_0=\NN\cup\{0\}$.
For a set $X$ and its subset $A$, ${\rm ch}_A$ denotes the characteristic function of $A$ on $X$.

Let $F$ be a non-archimedean local field, $\go$ its integer ring
and $\gp$ the maximal ideal of $\go$.
The cardinality of the residue field $\go/\gp$ is denoted by $q$.
We fix a uniformizer $\varpi$ of $F$.
For example, the $p$-adic field $F=\QQ_p$ for a prime number $p$ is a non-archimedean local field with characteristic $0$. In this case,
we have $\go=\ZZ_p$, $\gp=p\ZZ_p$ and $q=p$.
The prime number $p$ is a uniformizer of $\QQ_p$.

Set $G=\PGL(2, F)$ and $K=\PGL(2, \go)$. The unit element of $G$ is denoted by $1_2$.
Then $K$ is a maximal compact subgroup of $G$ and the Cartan decomposition
$G=\coprod_{n=0}^{\infty}K[\begin{smallmatrix}
\varpi^n & 0 \\ 0 & 1
\end{smallmatrix}]K$
holds (cf.\ \cite[Proposition 4.6.2]{Bump}).
We take a Haar measure $dg$ on $G$ such that the volume of $K$ equals $1$.

Let $B$ be the Borel subgroup of $G$ consisting of the upper triangular matrices.
For $s \in \CC$, we denote by $\pi_s={\rm Ind}_B^G(|\cdot|^{s}\boxtimes|\cdot|^{-s})$ the normalized parabolic induction whose representation space
$V_s$ consists of all $\CC$-valued locally constant functions $f$ on $G$ satisfying
$f([\begin{smallmatrix}
a & b \\ 0 & d
\end{smallmatrix}]g)=|\frac{a}{d}|^{s+1/2}f(g)$ for all $g \in G$ and $[\begin{smallmatrix}
a & b \\ 0 & d
\end{smallmatrix}] \in B$ and the $G$-action is the right translation. 
Set $\mathfrak{X}=i[0, 2\pi (\log q)^{-1}]$. If $s \in \mathfrak{X}$, then $\pi_s$ is an irreducible unitarizable representation of $G$
and the $K$-invariant subspace of $V_s$ is one dimensional.
We remark that the representations $\pi_s$ $(s\in \mathfrak{X})$ exhaust the unitary representations appearing in the spectral decomposition of $L^2(G/K)$.

Let $\mathcal{H}(G)$ be the Hecke algebra of $G$, that is, the space of all $\CC$-valued locally constant functions
with compact support. This is a non-commutative associative algebra over $\CC$ without unit,
where the product $f_1\circ f_2$ for $f_1, f_2 \in \Hcal(G)$ is the convolution defined
by
$$(f_1\circ f_2)(g)=\int_G f_1(gx^{-1})f_2(x)dx, \qquad g \in G.$$
Then $\Hcal(G)$ has a $\ast$-algebra structure by $f^*(g)=\overline{f(g^{-1})}$.
The $\ast$-algebra $\Hcal(G)$ acts on $C(G)$ by right translation $R$:
$$(R(f)\varphi)(g)=\int_G f(x)\varphi(gx)dx, \qquad f\in \Hcal(G), \, \varphi \in C(G).$$
Let $\Hcal(G, K)$ denote the subalgebra of $\Hcal(G)$
consisting of all $K$-biinvariant functions in $\Hcal(G)$.
It has a $\ast$-subalgebra of $\Hcal(G)$ and commutative (cf.\ \cite[Theorem 4.6.1]{Bump}).
Furthermore, $\Hcal(G, K)$ has a unit ${\rm ch}_{K}$.
Remark that $f(g)=f(g^{-1})$ for any $f \in \Hcal(G,K)$ because of $K[\begin{smallmatrix}
\varpi^{-n} & 0 \\ 0 & 1
\end{smallmatrix}]K = K[\begin{smallmatrix}\varpi^n & 0 \\ 0 & 1
\end{smallmatrix}]K$.
In particular, we have $f_1^*=\overline{f_1}$ and $f_1\circ f_2 = R(f_2)f_1$ for any $f_1, f_2 \in \Hcal(G,K)$.
The $L^2$-inner product on $L^2(G, dg)$ is denoted by $\langle \cdot, \cdot\rangle_G$.

%

\section{Hecke operators}\label{Hecke operators}

For any $n \in \NN_0$, let $T(\gp^n)$ be the characteristic function of
$$Z\backslash Z\{g\in {\rm M}_2(\go) \ | \ \det (g)\go=\gp^n \},$$
where $Z$ is the center of $\GL_2(F)$.
Then, $T(\gp^0) = {\rm ch}_K$ holds.
The action of $T(\gp^n)$ as $R(T(\gp^n))$ is called the Hecke operator at $\gp^n$. 
The following is a well-known recurrence equation (cf.\ \cite[Proposition 4.6.4]{Bump}).
\begin{lem}\label{relation of Hecke op}
	For any $n \in \NN$, we have $$T(\gp)\circ T(\gp^n)=T(\gp^{n+1})+q T(\gp^{n-1}).$$
\end{lem}

Let $\Psi_n$ be the characteristic function of $\bfK[\begin{smallmatrix}
\varpi^n & 0 \\ 0 & 1
\end{smallmatrix}]\bfK$ on $G$ for any $n\in \NN_0$.
Then we have the relation $T(\gp^n)=\sum_{r=0}^{\lfloor n/2 \rfloor}\Psi_{n-2r}$ in the same way as \cite[Lemma 6]{Sugiyama1}.
By solving this on $\Psi_n$, we have the following (cf.\ \cite[Remark 7]{Sugiyama1}).
\begin{lem}\label{psi}
We have $\Psi_0 = {\rm ch}_K$, $\Psi_1=T(\gp)$, and $\Psi_n = T(\gp^{n})-T(\gp^{n-2})$ for $n\ge 2$.
\end{lem}

Set $T'(\gp)=q^{-1/2}T(\gp)$ and $\Phi_n = q^{-n/2}\Psi_n$ ($n \in \NN_0$).
\begin{lem}\label{relation of phi}
We have the relations
$$T'(\gp)\circ \Phi_n=\Phi_{n+1}+\Phi_{n-1} \qquad (n \ge2),$$
$$T'(\gp)\circ \Phi_1=\Phi_2+\frac{q+1}{q}\Phi_0.$$
\end{lem}
\begin{proof}We use Lemmas \ref{relation of Hecke op} and \ref{psi}. For $n=1$, a direct computation gives us
$$T'(\gp)\Phi_1=\frac{1}{q}T(\gp)^2=\frac{1}{q}(T(\gp^2)+qT(\gp^0))=\frac{1}{q}(T(\gp^2)-T(\gp^0))
+\frac{1+q}{q}T(\gp^0)=\Phi_2+\frac{q+1}{q}\Phi_0.$$
For any $n \ge3$, we obtain
\begin{align*}
T'(\gp)\Phi_{n} = & q^{-1/2}T(\gp)\circ q^{-n/2}(T(\gp^n)-T(\gp^{n-2}))\\
= & q^{-(n+1)/2}(T(\gp^{n+1})+qT(\gp^{n-1}) -T(\gp^{n-1})-qT(\gp^{n-3})) \\
= & q^{-(n+1)/2}(T(\gp^{n+1})-T(\gp^{n-1}))+q^{-(n-1)/2}(T(\gp^{n-1})-T(\gp^{n-3})) \\
=& \Phi_{n+1}+\Phi_{n-1}.
\end{align*}
The case $n=2$ is similarly computed.
Thus we are done.
\end{proof}

Recall that $\Hcal(G,K)$ is a pre-Hilbert space with respect to the $L^2$-inner product $\langle \cdot,\cdot \rangle_G$.
The Cartan decomposition of $G$ yields that $\{\Phi_{n}\}_n$ is an orthogonal system
and generates $\Hcal(G,K)$ because of $f(g)=\sum_{n=0}^{\infty}f(\begin{smallmatrix}
\varpi^n & 0 \\ 0 & 1
\end{smallmatrix})\Psi_n$.

Therefore we can define $B^+$ and $B^{-}$ in $\End(\Hcal(G,K))$ by
$$B^+\Phi_0=\sqrt{\frac{q+1}{q}}\Phi_1, \qquad B^+\Phi_{n}=\Phi_{n+1} \quad(n\ge 1), $$
$$B^-\Phi_0=0, \qquad B^-\Phi_1= \sqrt{\frac{q+1}{q}}\Phi_0, \qquad B^-\Phi_n=\Phi_{n-1} \quad (n\ge2).$$
Set $B^0=0 \in \End(\Hcal(G,K))$. Combining these with Lemma \ref{relation of phi}, we have the following.

\begin{thm}\label{Hecke alg IFS}
The quintuple $(\Hcal(G,K), \{\Phi_n\}_n, B^+, B^-, B^0)$ is an interacting Fock space.
The associated Jacobi coefficient $(\{\omega_n\}_n, \{\a_n\}_n)$ is given by
$\omega_1=\frac{q+1}{q}$, $\omega_n=1$ for $n\ge 2$, and $\a_n=0$ for all $n \in \NN_0$. Furthermore, we have the quantum decomposition
$$R(T'(\gp))=B^{+}+B^{-}+B^{0} \in \End(\Hcal(G,K)).$$
\end{thm}
The operators $B^+$, $B^-$ and $B^0$ are called the creation, the annihilation, and the preservation operators, respectively.

Let $\varphi : \Hcal(G, K)\rightarrow \CC$ be the vector state for $\Phi_0$:
$$\varphi(f)=\langle R(f)\Phi_0, \Phi_0\rangle_G, \qquad f \in \Hcal(G, K).$$
Then $(\Hcal(G,K), \varphi)$ is a quantum probability space and $T'(\gp)$ is a real random variable
in $(\Hcal(G, K), \varphi)$.
This yields the following equality of $m$-th moments.
\begin{cor}\label{moment}
	For any $m\in\NN_0$, we have
$$\varphi(T'(\gp)^m) = \langle R(T'(\gp))^m\Phi_0, \Phi_0\rangle_G=\int_{-\infty}^{\infty}x^m \frac{q+1}{(q^{1/2}+q^{-1/2})^2-x^2}\frac{\sqrt{4-x^2}}{2\pi}{\rm ch}_{[-2,2]}(x)dx.$$
\end{cor}
\begin{proof}
This is a consequence of \cite[Theorem 4.11]{Obata}. To complete the proof, it suffices to determine a probability measure $\mu$ on $\RR$ associated with $(\{\omega_n\}_n, \{\a_n\}_n)$.
In this case, as in \cite[p.96]{Obata},
$\mu$ is the free Meixner distribution with parameter $(\frac{q+1}{q}, 1,0)$, i.e.,
$$\int_{-\infty}^{\infty}\frac{1}{z-x}d\mu(x)=\cfrac{1}{z-\cfrac{\frac{q+1}{q}}{z-\cfrac{1}{z-\cfrac{1}{\ddots}}}}, \qquad {\rm Im}(z)\neq 0.$$
Since the continued fraction as above is periodic,
the right-hand side is computed as
$$\frac{(2-\frac{q+1}{q})z-\frac{q+1}{q}\sqrt{z^2-4}}{2(1-\frac{q+1}{q})z^2+2(\frac{q+1}{q})^2},$$
where the square root is taken so that $\sqrt{z}>0$ for any $z >0$.
Hence the inverse of the Stieltjes transformation yields
\begin{align}\label{explicit mu}d\mu(x) = \frac{q+1}{(q^{1/2}+q^{-1/2})^2-x^2}\frac{\sqrt{4-x^2}}{2\pi}{\rm ch}_{[-2,2]}(x)dx.
\end{align}
Hence, we obtain the desired formula.
\end{proof}

There exists an orthogonal polynomial system $\{P_n\}_n$ associated with $(\{\omega_n\}_n, \{\a_n\}_n)$.
By the proof of \cite[Theorem 4.11]{Obata}, we have a unitary mapping $\Ucal : \Hcal(G,K) \rightarrow L^2(\RR, d\mu)$ by
$\Phi_n \mapsto (\prod_{j=1}^n \omega_j)^{-1/2}P_n(x)$
and $R(T'(\gp))$ corresponds to the $x$-multiple on $L^2(\RR, d\mu)$ by $\Ucal$.
The mapping $\Ucal$ induces an isometry $L^2(K\bsl G/K, dg)\cong L^2(\RR, d\mu)$.

\smallskip
\noindent
{\bf Remark }:
In the proof of Corollary \ref{moment}, we can determine $\mu$ also in the following way.
By Lemma \ref{relation of Hecke op}, $T(\gp^n)$ is expressed by $T'(\gp)$ as $q^{-n/2}T(\gp^n)=U_n(2^{-1}T'(\gp))$ for any $n\in \NN_0$, where $U_n$ is the $n$-th Chebyshev polynomial of 2nd kind.
Combining this with the definition of $\Phi_n$, the family $\{P_n\}_n$ is described as $P_0(x)=1$, $P_1(x)=U_1(2^{-1}x)$ and $P_n(x) = U_n(2^{-1}x)-q^{-1}U_{n-2}(2^{-1}x)$ for $n\ge 2$.
Then, from the computation in \cite[\S2.3]{Serre},
$\mu$ is given by \eqref{explicit mu}.

\section{Spherical functions and Fourier transforms}
In this section, we show an application of Corollary \ref{moment}
to the Fourier inversion formula for $G/K$.

\subsection{Fourier inversion formula}
Let us review the Fourier inversion formula for $G$.
For $s \in \mathfrak{X}$, there exists $\Omega_s : G \rightarrow \CC$
with properties
\begin{enumerate}
	\item[(i)] $R(T'(\gp))\Omega_s=(q^{-s}+q^{s})\Omega_s$, 
	\item[(ii)] $\Omega_s(k_1gk_2)=\Omega_s(g)$ for all $k_1, k_2 \in K$ and $g \in G$, 
	\item[(iii)] $\Omega_s(1_2)=1$.
\end{enumerate}
Then $\Omega_s$ is uniquely determined by three properties, and called a spherical function.
The spherical function $\Omega_s$ can be given as follows.
Let $\phi_0$ be the $K$-invariant vector of $\pi_s$ determined by $\phi_0\in V_s^K$ and $\phi_0(1_2)=1$.
The representation $\pi_s$ has a $G$-invariant inner product given by
$\langle \phi_1, \phi_2 \rangle_K = \int_{K}\phi_1(k)\overline{\phi_2(k)}dk$ ($\phi_1, \phi_2 \in V_s$).
Note $\|\phi_0\|_K=1$.
Then, $\Omega_s$ is expressed as $\Omega_s(g)=\langle\pi_s(g)\phi_0, \phi_0 \rangle_K$ (cf.\ \cite[Proposition 4.6.6]{Bump}).
In particular, $s\mapsto \Omega_s(g)$ for any fixed $g\in G$ is analytic and $|\Omega_s(g)| \le \|\pi_s(g)\phi_0\|_K \|\phi_0\|_K=\|\phi_0\|_K^2=1$ holds.

We define the Fourier transformation $\Fcal : \Hcal(G, K) \rightarrow C(\mathfrak{X})$ by
$$\Fcal(f)(s)=\int_{G}\Omega_s(g)f(g)dg.$$
Let $\Acal=\CC[q^{-s}+q^s] \subset C(\mathfrak{X})$ be the ring of polynomial functions on $\mathfrak{X}$ with respect to $q^{-s}+q^s$.
Similarly to $\Fcal$, we define $\Fcal^{*}: \Acal \rightarrow C(G)$ by
$$\Fcal^*(\a)(g)=\int_{\gX}\Omega_s(g)\a(s)d\nu(s).$$
Here $d\nu(s)$ is a positive Radon measure on $\gX$ given by
$$d\nu(s)=\frac{1+q^{-1}}{4\pi i}\frac{|1-q^{-2s}|^2}{|1-q^{-2s-1}|^2}(\log q) ds.$$
This measure is called the spherical Plancherel measure for $G$.
Remark that the space $\mathfrak{X}$ parametrizes the spherical tempered unitary dual of $G$.
The Fourier inversion formula is stated as follows.

\begin{thm}\label{Fourier inversion}
The image of $\Fcal$ and of $\Fcal^*$ is contained in $\Acal$ and $\Hcal(G,K)$, respectively. 
Both $\Fcal:\Hcal(G,K)\rightarrow \Acal$ and $\Fcal^*:\Acal\rightarrow \Hcal(G,K)$ are isometries and each of these is the inverse
of the other.
Furthermore, both $\Fcal$ and $\Fcal^*$ are extended to isometries between $L^2$-spaces;
$$L^2(K \backslash G / K, dg) \cong L^2(\gX, d\nu).$$
\end{thm}
This result as above is a special case of \cite[Theorem 5.1.2]{Macdonald} (see also \cite[Theorem 4.7]{Tadic} and \cite[Theorem 3.1]{Tsuzuki2}).

\subsection{Proof of the Fourier inversion formula}
By Corollary \ref{moment}, we obtain the isometry $\Ucal : \Hcal(G,K)\rightarrow \bigoplus_{n=0}^{\infty}\CC P_n$, as was explained in the last of \S \ref{Hecke operators}.
We have a homeomorphism $\mathfrak{X}\cong [-2,2]$ by $s\mapsto q^{-s}+q^s$,
under which $d\nu$ is transformed into $d\mu$ given by \eqref{explicit mu}.
This leads us naturally to the isometry $\iota : L^2(\RR, d\mu)\cong L^2(\mathfrak{X}, d\nu)$.

We give integral representations of $\Ucal$ and of $\Ucal^{-1}$.
The following gives us an alternative proof of Theorem \ref{Fourier inversion}.
\begin{thm}
We have $\iota\circ \Ucal=\Fcal$ and $\Ucal^{-1}\circ\iota^{-1}=\Fcal^*$.
Namely, we have
${\rm Im}(\Fcal)\subset \Acal$, ${\rm Im}(\Fcal^*)\subset \Hcal(G,K)$,
$$\iota\circ\Ucal(f)= \int_{G}\Omega_s(g)f(g)dg$$
for any $f \in \Hcal(G,K)$ and
$$\Ucal^{-1}\circ\iota^{-1}(\a)= \int_{\mathfrak{X}}\Omega_s(g)\a(s)d \nu(s)$$
for any $\a \in \Acal$.
\end{thm}
\begin{proof}
We have only to compute $\iota \circ \Ucal(T'(\gp)^m)$ and $\Ucal^{-1}\circ \iota^{-1}((q^{-s}+q^s)^m)$ for any $m\in \NN_0$.
By definition, we see
\begin{align*}
\iota\circ\Ucal(T'(\gp)^m)(s)=(q^{-s}+q^{s})^m= R(T'(\gp)^m)\Omega_s(1_2)=\int_{G}\Omega_s(1_2g)T'(\gp)^m(g)dg
\end{align*}
for $m\in \NN_0$, which leads the assertion for $\iota\circ \Ucal$.
Let $\langle\cdot, \cdot \rangle_\mathfrak{X}$ denote the $L^2$-inner product on $L^2(\mathfrak{X}, d\nu)$.
For $m, n \in \NN_0$, we see
\begin{align*}
& \langle \Ucal^{-1}\circ\iota^{-1}((q^{-s}+q^s)^m)-\int_{\mathfrak{X}}\Omega_s(\cdot)(q^{-s}+q^s)^md\nu(s), T'(\gp)^n\rangle_G\\
= & \langle \Ucal^{-1}\circ\iota^{-1}((q^{-s}+q^s)^m), T'(\gp)^n\rangle_G
-\langle \int_{\mathfrak{X}}\Omega_s(\cdot)(q^{-s}+q^s)^md\nu(s), T'(\gp)^n\rangle_G\\
= & \langle (q^{-s}+q^s)^m, \iota\circ\Ucal(T'(\gp)^n)\rangle_\mathfrak{X}
-\int_{G}\int_{\mathfrak{X}}\Omega_s(g)(q^{-s}+q^s)^m d\nu(s)T'(\gp)^n(g)dg.
\end{align*}
Here we use that both $\Ucal$ and $\iota$ are isometries. 
By the integral representation of $\iota\circ \Ucal$, the first term of the last line of the equalities above
is transformed into
$$\int_\mathfrak{X} (q^{-s}+q^s)^m \int_G \Omega_s(g)T'(\gp)^n(g)dg d\nu(s).$$
Since both $\mathfrak{X}$ and ${\rm supp}(T'(\gp)^n)$ are compact, Fubini's theorem is applied
and the last line of the equalities above vanishes.
Since $n$ is arbitrary, the equality
$$\Ucal^{-1}\circ \iota^{-1}((q^{-s}+q^s)^m)(g)=\int_{\mathfrak{X}}\Omega_s(g)(q^{-s}+q^s)^md\nu(s)$$
holds,
from which we obtain the assertion for $\Ucal^{-1}\circ \iota^{-1}$.
\end{proof}

\smallskip
\noindent
{\bf Remark }:
By the Macdonald formula \cite[Theorem 4.1.2]{Macdonald}, we have 
\begin{align}\label{explicit of omega}
\Omega_s(\begin{smallmatrix}
	\varpi^n & 0 \\ 0 & 1
\end{smallmatrix})=\frac{q^{-n/2}}{1+q^{-1}}\left(\frac{1-q^{-1+2s}}{1-q^{2s}}q^{-ns}+\frac{1-q^{-1-2s}}{1-q^{-2s}}q^{ns}\right), \qquad n\in \NN_0
\end{align}
(see also \ \cite[Theorem 4.6.6]{Bump}).
His proof of Theorem \ref{Fourier inversion}, which is valid for a general $p$-adic group, relies on explicit calculations of the integrals
$\Fcal(\Psi_m)$ and $\Fcal^*((q^{-s}+q^{s})^m)$ (cf.\ \cite{Tadic} and \cite[\S5.4]{Tsuzuki2}).
Contrary to Macdonald's method, we use properties (i) and (iii) of the spherical function $\Omega_s$ but not its explicit formula \eqref{explicit of omega}. Indeed, we first obtain isometries $\iota \circ \Ucal$
and
$\Ucal^{-1}\circ \iota^{-1}$ by virtue of quantum probability theory,
and finally obtain integral representations $\Fcal$ and $\Fcal^*$ of them.

\section*{Acknowledgements}

The authors would like to thank Hiroki Sako for fruitful discussion.
Takehiro Hasegawa was partially supported by JSPS KAKENHI (grant number 15K17508).
Hayato Saigo was partially supported by JSPS KAKENHI (grant number 26870696).
Seiken
Saito was partially supported by JSPS KAKENHI (grant number 16K05259).


\medskip
\noindent
{Takehiro HASEGAWA \\
	Shiga University, Otsu, Shiga 520-0862, Japan} \\
{\it E-mail} : {\tt thasegawa3141592@yahoo.co.jp}

\medskip
\noindent
{Hayato SAIGO\\
	Nagahama Institute of Bio-Science and Technology, 1266, Tamura,
	Nagahama 526-0829, Japan} \\
{\it E-mail} : {\tt h\_saigoh@nagahama-i-bio.ac.jp}

\medskip
\noindent
{Seiken SAITO\\
	Waseda University, Shinjuku, Tokyo 169-8050, Japan} \\
{\it E-mail} : {\tt seiken.saito@aoni.waseda.jp}

\medskip
\noindent
{Shingo SUGIYAMA\\
	Institute of Mathematics for Industry, Kyushu University, 744, Motooka, Nishi-ku, Fukuoka 819-0395, Japan} \\
{\it E-mail} : {\tt s-sugiyama@imi.kyushu-u.ac.jp}

\end{document}